\documentclass[12pt]{article}
\usepackage{amsmath,amssymb,amsthm}
\usepackage{graphicx}
\usepackage{setspace}
\usepackage{mathtools}
\usepackage{tabu} 
\usepackage[utf8]{inputenc}
\usepackage[english]{babel}

\makeatletter
\renewcommand*{\@biblabel}[1]{\hfill#1.}
\makeatother

\theoremstyle{plain}
\newtheorem{theorem}{Theorem}
\newtheorem*{theorem*}{Theorem}
\newtheorem{lemma}{Lemma}

\newtheorem*{corollary*}{Corollary}

\newtheorem*{remark*}{Remark}

\begin{document}

\title{\bf Fitting rectangles under vulnerability curves: optimal water flow through plants}

\date{}

\author{Sergiy Koshkin, Michael Tobin, Jeffae Schroff,\\
Matt Capobianco, Sarah Oldfield}

\maketitle

\abstract{We study an optimization problem for a model of steady state water transport through plants that maximizes water flow subject to the constraints on hydraulic conductance due to vulnerability to embolism (air blockage of conduits). The model has an elementary geometric interpretation, and exhibits bottleneck behavior where one of the plant segments limits the overall optimal flow, sometimes in a counterintuitive way. The results show good agreement with experimental measurements and provide support for the hypothesis that leaves serve as a safety buffer protecting stems against excessive embolism.}

%%%%%%%%%%%%%%%%  Introduction Section %%%%%%%%%%%%%%%%
\section*{Introduction}

According to the cohesion-tension theory \cite[Ch.3]{TZ}, plants absorb water from the soil and transport it by creating a negative pressure in their leaves lower than in the soil. Plants regulate this pressure by changing the aperture size of stomata, small openings that allow water vapor loss from leaves. Water transport to leaves is essential to a number of plants' vital functions. For instance, carbon gain from atmospheric CO$_2$ optimized in recent literature depends monotonically on the flow \cite{Sper,Wolf}. Therefore, one can expect that the flow would be optimized with respect to several constraints imposed by plants' overall structure and operation. 

Modeling water transport has to take into account that the transport system in plants is not made of continuous pipes, but is a network of conduits formed by the cell walls of single or multiple cells that have died. These conduits are connected by pits in the secondary cell walls where only a pit membrane with small pores separates the conduits from each other. At any time, some conduits in the network are filled with air and do not transport water. When large negative pressures develop in a water filled conduit, air can be pulled through a pore in the pit membrane into the functioning conduit, filling it with air and preventing further water transport through the conduit. These newly embolized conduits reduce the overall hydraulic conductance of the network. 

To restore conductance, the plant would have to refill the air-blocked conduit with water. But this embolism repair is complicated by the fact that the surrounding water filled conduits often remain under negative pressure throughout the diurnal cycle during much of the year \cite{Hol,Wh}. Refilling of conduits while the transport system remains under negative pressure is an area of active research and debate \cite{ZwH,Vent}.  
\begin{figure}[!ht]\label{diagram}
\begin{center}
\includegraphics[width = 0.6\textwidth]{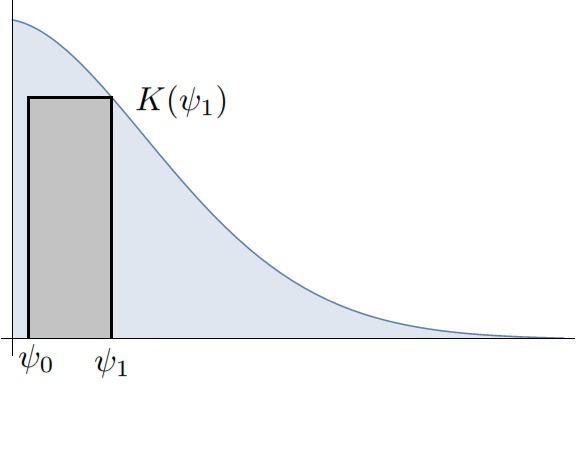}
\end{center}
\vspace{-0.15\textwidth}
\caption{\label{VulCurv} Typical vulnerability curve and a flow rectangle.}
\end{figure}

In our optimality model we assume that conductance is not restored on the relevant (seasonal) time scales. We also neglect water losses and storage, and schematize a plant as a chain of conducting segments with the entire canopy of leaves lumped into a single segment at the top. While admittedly crude, such simplifications are not uncommon in the literature, and for some tree species the measured midday pressures are in a surprisingly good agreement with the predicted optimal values (Section \ref{S5}). Moreover, they lead to an interesting mathematical problem with solution structure that illuminates the distribution of hydraulic conductance between stem and leaves in a plant, which is an active research topic in plant biology, see \cite{Drake, Hao, John, Sack}. The idea that leaves serve as a safety buffer for preventing embolism damage to stems was first hypothesized in \cite{Zim}, and gained broad support in recent years, although now this function is largely attributed to living cells rather than to the embolism of conduits \cite{Scof1}. The rationale would be that leaves are less costly to grow for a plant, and therefore are more disposable. Earlier works, however, produced seemingly contrary results by only comparing maximal leaf and stem hydraulic conductances. For example, a study of $34$ species in \cite{Sack} noted that they are as $3:1$ on average. Part of the explanation is provided by more comprehensive comparisons, see Section \ref{S5}.

The optimality problem is formulated mathematically in Section \ref{S0}. We consider the two segment case in Sections \ref{S1}-\ref{S3} and prove existence and uniqueness results that allow us to formulate a computational algorithm for solving it. The problem for three or more segments is treated in Section \ref{S4}. We apply our algorithm to the data from four plant species and draw conclusions in Section \ref{S5}.

\section{Mathematical model}\label{S0}

To state the problem mathematically it is convenient to change the sign of negative pressures and represent them as a difference of {\it water potentials}. The potential is thus the difference between the normal atmospheric pressure, taken as the baseline, and a pressure in the plant's water column. Under our convention water flows from low to high potential. The Darcy's law then states that the steady state flow of water $F$ through a segment is proportional to the difference $\Delta\psi$ of the water potentials at its ends, $F=K\Delta\psi$, where $K$ is a positive proportionality constant called conductance of the segment. If the plant is modeled by a single segment with $\psi_L$ the water potential in the leaf and $\psi_0$ the water potential in the soil (or root) then $F=K(\psi_L-\psi_0)$. More generally, the plant is split into a series of segments with conductances $K_1,K_2,K_3,...$ and water potentials $\psi_0, \psi_1,\psi_2\,...$ at the nodes where the segments link. Then we get for the steady state flow: 
$$
F=K_1(\psi_1-\psi_0)=K_2(\psi_2-\psi_1)=\dots\,.
$$
The conductances $K_i$ depend on water potentials in both the current and the previous segments, so the problem is nonlinear. 

The highest level of conductance sustainable under water potential $\psi$ is given by a function $K(\psi)$ known as the {\it vulnerability curve}. Experimental procedures for measuring  vulnerability curves are discussed in \cite{Coch}, and they are often fit to Weibull functions $K(\psi)=K_{\max}e^{-\big(\frac{\psi}{p}\big)^\nu}$, where $K_{\max}$ is the maximal conductance in the absence of embolism, see \cite{Ogle}. A typical shape is shown on Fig.\,\ref{VulCurv}. Finally, no restoration of conductance  implies that vulnerability curves function irreversibly, as in magnetic hysteresis: once exposure to potential $\psi_i$ lowered the conductance of the $i$-th segment to $K_i(\psi_i)$ it remains that even for lower potentials. Mathematically, this means that the steady state flow through a segment is $K_i(\psi_i)(\psi_i-\psi_{i-1})$ rather than $\int_{\psi_{i-1}}^{\psi_i}K_i(\psi)\,d\psi$. Thus, we can state the problem as follows: 
\smallskip

\noindent {\bf Optimization problem.} Given the soil potential $\psi_0$ find potentials $\psi_1,\dots,\psi_n:=\psi_L$ to maximize the flow:
\begin{equation}\label{GenProb2}
F=K_1(\psi_1)(\psi_1-\psi_0)=K_2(\psi_2)(\psi_2-\psi_1)=\dots=K_n(\psi_n)(\psi_n-\psi_{n-1})
\end{equation}
\begin{figure}[!ht]
\begin{center}
\includegraphics[width = 0.7\textwidth]{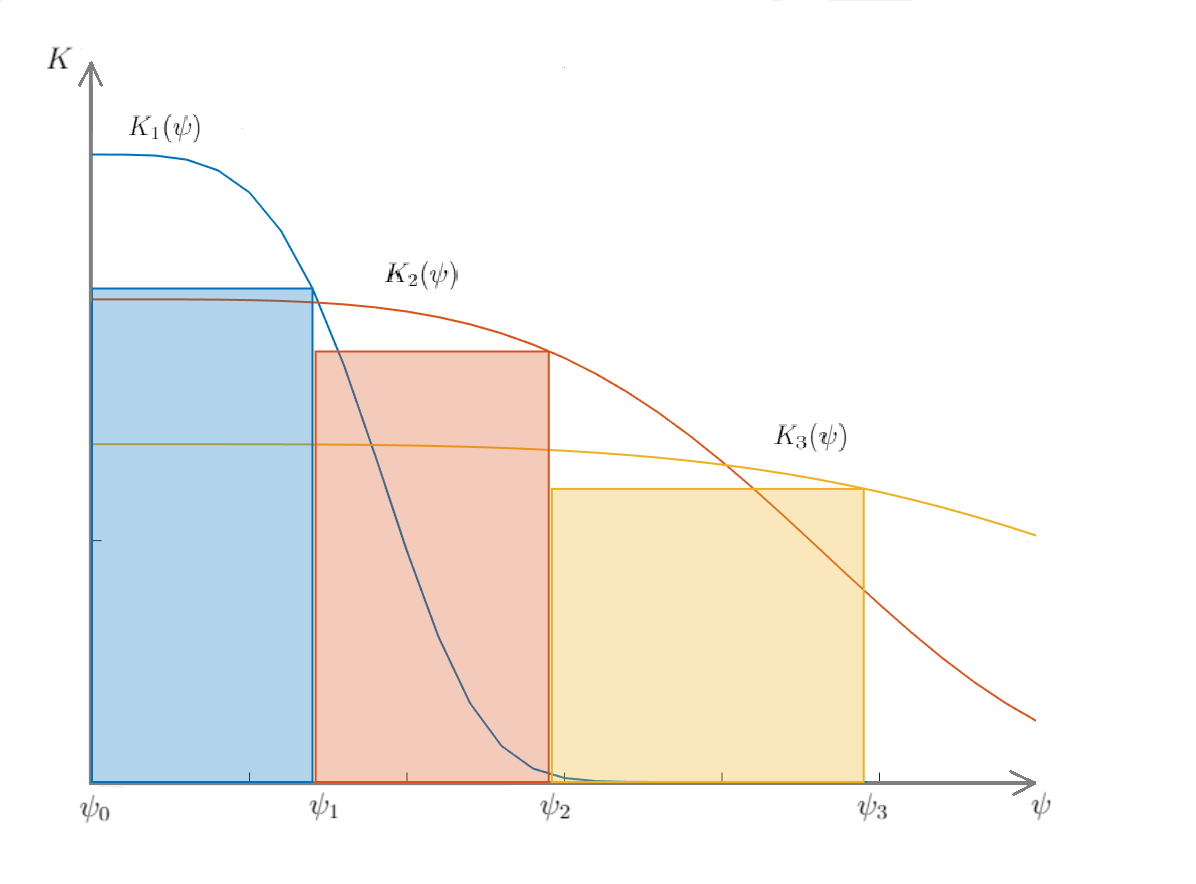}
\end{center}
\vspace{-0.1\textwidth}
\caption{\label{FlowCurvs} Optimal flow rectangles under vulnerability curves.}
\end{figure}
Geometrically, the objective is to maximize the (common) area of rectangles lined up next to each other along the $x$-axis, each one tightly fit under the graph of (its own) function, see Fig. \ref{FlowCurvs}, where the case of two stem segments and one leaf segment is shown. For two segments and piecewise-linear vulnerability curves the problem was studied numerically in \cite{JS}.

Solutions to this problem exhibit bottleneck behavior (see Theorem 4), namely the optimal flow is suboptimal when restricted to an initial chain of segments taken in isolation. This is because the segment following it (the bottleneck) limits the flow size. This suggests testing whether the leaf segment produces the bottleneck to decide whether leaves serve as a safety buffer. This test is more comprehensive than comparing either maximal conductances \cite{Sack}, or stem and leaf water potentials at $50\%$ loss of conductance \cite{Hao}, and may produce answers at variance with them. Computations for some species indicate that leaves can become the bottleneck and limit stem's exposure to high potentials even if the maximal leaf conductance is vastly (over 30:1) larger than the maximal stem conductance. 

\section{Optimal solution for two segments: existence}\label{S1}

When there is only one segment the problem reduces to maximizing a one-variable function $F(\psi_1)=K_1(\psi_1)(\psi_1-\psi_0)$, which reduces to solving a single equation $F'(\psi_1)=0$ when $F$ is differentiable. We will analyze the case of two segments in detail first because it is less cumbersome but already displays the main features of the general case. In this section we will prove existence of an optimal solution under some very mild conditions. Later, under somewhat stronger conditions, we will develop an algorithm for finding the solution(s). 

Let us begin by simplifying the notation somewhat. Namely, we replace $\psi_i$ with $x_i$ and set $f(x):=K_1(x)$, $g(x):=K_2(x)$ to be the vulnerability curves of the stem and the leaf, respectively. They are defined on $[0,\infty)$, non-negative and continuous. Let us also introduce the flow functions 
\begin{equation}\label{FlowFun}
F(x_1):=(x_1-x_0)f(x_1);\ \ \ \ \ G(x_1,x_2):=(x_2-x_1)g(x_2).
\end{equation}

Assume that some $x_0\geq0$, $f$ and $g$ are given. For two segments the optimization problem can now be restated in the standard form: 
\begin{equation}\label{TwoProb}
\text{Maximize $G(x_1,x_2)$ subject to the constraint $F(x_1)=G(x_1,x_2)$.}
\end{equation}
The proof of existence will be based on the Weierstrass's theorem that a continuous function on a closed bounded set always attains its maximal value. If $f$ and $g$ are continuous then so are $F$ and $G$, therefore the constraint set 
\begin{equation}\label{ConSet}
Con:=\{(x_1,x_2)\in\mathbb{R}^2\,\big|\,x_1,x_2\geq0\text{ and }F(x_1)=G(x_1,x_2)\}
\end{equation}
is closed. It is trivially nonempty since we can take $x_2=x_1=x_0$ for the flow of $0$. However, a simple continuity argument shows that we can do better.
\begin{lemma}\label{IVT}
Let $x_0\geq0$, and the functions $f,g$ be non-negative and continuous with $f(x_0),g(x_0)>0$. Then there exist $x_1<x_2$ such that $F(x_1)=G(x_1,x_2)\geq\varepsilon>0$.
\end{lemma}
\begin{proof} Choose $x_2>x_0$ close enough so that $f(x),g(x)$ are strictly positive on $[x_0,x_2]$, and consider $h(x):=F(x)-G(x,x_2)$. Then $h(x_0)=(x_2-x_0)g(x_2)>0$ and $h(x_2)=-(x_2-x_0)f(x_2)<0$. By the Intermediate Value Theorem, there is $x_1\in(x_0,x_2)$ where $h(x_1)=0$, i.e. $F(x_1)=G(x_1,x_2)=(x_1-x_0)f(x_1)>0$.\,$\square$
\end{proof}
Unfortunately, not only can we not expect that $Con$ is bounded but without additional conditions on $f,g$ we can not even expect that an optimal solution exists. If $x_0=1$ and $f(x)=\frac1{\sqrt{x}}$ for $x\geq1$ then $F(x_1)=\frac1{\sqrt{x_1}}(x_1-1)\xrightarrow[x_1\to\infty]{}\infty$, i.e. the flow can be increased indefinitely by raising $x_1$. If $g(x)$ behaves similarly we can match $G(x_1,x_2)$ to arbitrarily large values as well, and there will be no maximum. Fortunately, realistic vulnerability curves decrease much faster, which rules out such behavior.
\begin{lemma}\label{Comp}
Suppose $xg(x)\xrightarrow[x\to\infty]{}0$. Then the maximum of $G(x_1,x_2)$ is attained on $Con$ if and only if it is attained on $Con_{\varepsilon}:=Con\,\cap\{(x_1,x_2)\in\mathbb{R}^2\,\big|\,G(x_1,x_2)\geq\varepsilon\}$ with $\varepsilon$ from Lemma \ref{IVT}, which is bounded.
\end{lemma}
\begin{proof} By Lemma \ref{IVT} there exist $(x_1,x_2)\in Con$ such that $G(x_1,x_2)\geq\varepsilon$. Clearly, the maximum on $Con$ can not be less, and therefore is attained on $Con_{\varepsilon}$. But by assumption
$$
G(x_1,x_2)=(x_2-x_1)g(x_2)\leq x_2g(x_2)\xrightarrow[x\to\infty]{}0\,.
$$
Therefore, $G(x_1,x_2)<\varepsilon$ whenever $x_2>R$, where $R$ does not depend on $x_1$. Moreover, if $x_1>x_2$ then $G(x_1,x_2)\leq0$. Thus, if $G(x_1,x_2)\geq\varepsilon$ then $0\leq x_1<x_2\leq R$, and $Con_{\varepsilon}$ is bounded.\,$\square$
\end{proof}
Existence of an optimal solution is now a direct consequence of our lemmas and the Weierstrass's theorem. 
\begin{theorem}\label{Exist}
Suppose $x_0\geq0$, and the functions $f,g$ be non-negative and continuous. Suppose $f(x_0),g(x_0)>0$ and $xg(x)\xrightarrow[x\to\infty]{}0$. Then there exists a solution to the optimization problem \eqref{TwoProb} and the maximum is strictly positive.
\end{theorem}

\section{Optimal solution for two segments: uniqueness and the bottleneck}\label{S2}

In general, we can not expect that the optimal solution is unique. However, in many cases the vulnerability curves $f,g$ have properties that ensure at least ``partial" uniqueness. Recall that a function is called {\it unimodal} if it has a single maximum \cite{Uni}. It turns out that the flow functions like $F(x):=(x-x_0)f(x)$ are often unimodal for any $x_0\geq0$. When the flow functions are unimodal our problem satisfies a condition that can be called ``optimization by parts": if the flow is maximized overall then it is maximized on at least one of the segments separately. 
\begin{theorem}[{\bf Optimization by parts}]\label{OptPart}
In conditions of Theorem \ref{Exist} let $F$ and $G(x_1,\cdot)$ be unimodal for any $x_0,x_1\geq0$, and let $(x_1^*,x_2^*)$ be an optimal solution. Then $x_1^*$ is a maximizer of $F$ on $[x_0,\infty)$, or $x_2^*$ is a maximizer of $G(x_1^*,\cdot)$ on $[x_1^*,\infty)$. In the second case the solution is unique. 
\end{theorem}
\begin{proof} Suppose $x_1^*$ is not a maximizer of $F$. If $\displaystyle{\max_{x_2}G(x_1^*,x_2)>G(x_1^*,x_2^*)}$ then there is $\widetilde{x}_2$ such that $G(x_1^*,\widetilde{x}_2)>G(x_1^*,x_2^*)$. By continuity, changing $x_1^*$ slightly in either direction would still yield $G(\widetilde{x}_1,\widetilde{x}_2)>G(x_1^*,x_2^*)$. But since $F$ is unimodal and $x_1^*$ is not the maximizer there is such a choice $\widetilde{x}_1$ that makes $F(\widetilde{x}_1)>F(x_1^*)$. Therefore, $(x_1^*,x_2^*)$ can not be an optimal solution unless $x_2^*$ maximizes $G(x_1^*,\cdot)$.

Let $(\widetilde{x}_1^*,\widetilde{x}_2^*)$ be another optimal solution. If either $\widetilde{x}_1^*$ or $x_1^*$ is a maximizer of $F$ then since $F(\widetilde{x}_1^*)=F(x_1^*)$ and $F$ is unimodal we have $\widetilde{x}_1^*=x_1^*$. If not, we may assume without loss of generality that $\widetilde{x}_1^*\geq x_1^*$. Since $x_2^*$ is a maximizer of $G(x_1^*,\cdot)$ we have
\begin{equation}\label{MaxDecr}
G(\widetilde{x}_1^*,\widetilde{x}_2^*)=(\widetilde{x}_2^*-\widetilde{x}_1^*)g(\widetilde{x}_2^*)
\leq(\widetilde{x}_2^*-x_1^*)g(\widetilde{x}_2^*)\leq\max_{x_2}G(x_1^*,x_2)=G(x_1^*,x_2^*)
\end{equation}
Since the first and the last values are equal both inequalities are equalities, and since by Theorem \ref{Exist} $g(\widetilde{x}_2^*)>0$ we must have $\widetilde{x}_1^*=x_1^*$. Then $\widetilde{x}_2^*=x_2^*$ follows from the unimodality of $G(x_1^*,\cdot)$.\,$\square$
\end{proof}
It is easy to construct examples where $F$ and $G$ are unimodal, $x_1^*$ is the maximizer for $F$, but $x_2^*$ is not unique, indeed this happens generically. This is because neither $x_2^*$ nor $\widetilde{x}_2^*$ have to be maximizers and we should typically find at least two equal flow values on different sides of the maximizers. But it is also clear that the smallest is the one that should be chosen because it minimizes exposure of the leaf segment to high water potentials. So for all practical purposes the solution is unique in both cases. The following lemma gives a simple sufficient condition on a flow function to be unimodal. One can easily check that the functions of interest to us, $f(x)=K(1-\frac{x}{p})$ on $[0,p)$ and $f(x)=Ke^{-(x/p)^\nu}$ on $[0,\infty)$, with $K,p>0$ and $\nu\geq1$, satisfy it. So do exponential-sigmoid functions also used to fit vulnerability curves \cite{Ogle}.
\begin{lemma}\label{Unimod}
Let $f$ be a strictly positive, monotone decreasing differentiable function on $[0,a)$ (possibly $a=\infty$). Suppose also that $xf(x)\xrightarrow[x\to a]{}0$ and $\ln\frac1f$ is convex down on $[0,a)$. Then for any $x_0\in[0,a)$ if $f(x_0)>0$ then $F(x):=(x-x_0)f(x)$ is unimodal on $[x_0, a)$.
\end{lemma}
\begin{proof} We have $F(x_0)=0$, $F(x)\leq xf(x)\xrightarrow[x\to a]{}0$, and $F(x)$ is positive on $[x_0,a)$, so it must have a maximum there. Since $f$ is differentiable at a point of maximum $F'(x)=f(x)+f'(x)(x-x_0)=0$, or equivalently, 
$$
x-x_0=-\frac{f'(x)}{f(x)}=\frac1{(\ln\frac1f)'(x)}.
$$
By conditions on $f$ the right hand side is non-negative, and since $\ln\frac1f$ is convex down $(\ln\frac1f)'$ is monotone increasing. Hence its reciprocal is monotone decreasing. Since $x-x_0$ is strictly monotone increasing they can be equal at at most one point, the maximum. 
\end{proof}

A simple argument shows that two distinct cases arise even without assuming unimodality. Suppose the maximal value of $F(x_1)$ on $[x_0,\infty)$ is attained at $x_1^*$, it gives the flow carrying capacity of the stem segment. If the leaf segment can transport this much flow, i.e. $G(x_1^*,x_2)=F(x_1^*)$ has a solution $x_2^*$ on $[x_1^*,\infty)$, then $(x_1^*,x_2^*)$ solve the original optimization problem. Let us call this the {\it non-bottleneck case}. If the maximal value of $G(x_1^*,x_2)$ is strictly less than $F(x_1^*)$ the leaf serves as the water flow {\it bottleneck}, it limits the overall carrying capacity of the plant. 

\section{Optimal solution for two segments: algorithm}\label{S3}
 
The existence proof we gave above was non-constructive. To find a solution assume that $f$, and therefore $F$, is differentiable. Then solving the non-bottleneck case reduces to solving two non-linear equations, the first one being $F'(x_1)=0$, and the second one $G(x_1^*,x_2)=F(x_1^*)$. It is possible that one or both equations have multiple solutions, in which case the solution with minimal $x_i^*$ should be selected to reduce exposure to high water potentials. 

If $g$, and then $G$, is also differentiable, by optimization by parts in the bottleneck case the solution satisfies $\frac{\partial G}{\partial x_2}=0$. This equation should be solved in conjunction with the constraint $F(x_1)=G(x_1,x_2)$. Unfortunately, the equations do not decouple in this case and we have to solve them as a non-linear system. However, $\frac{\partial G}{\partial x_2}=0$ can be explicitly solved for $x_1$, namely $x_1=x_2+\frac{g(x_2)}{g'(x_2)}$, which upon substitution into the constraint equation produces a single non-linear equation for $x_2$. Note that if $g$ satisfies the conditions of Lemma \ref{Unimod} then the right hand side is a monotone increasing function of $x_2$, i.e. solution with the smallest $x_2$ will also have the smallest $x_1$. Our considerations can be summarized as follows.
\begin{theorem}\label{TwoPart} Let $x_0\geq0$ and $f,g$ be non-negative and differentiable. Then any optimal solution $(x_1^*,x_2^*)$ solves at least one of the equations $F'(x_1^*)=0$ (non-bottleneck case) or $\frac{\partial G}{\partial x_2}(x_1^*,x_2^*)=0$ (bottleneck case), in addition to the constraint $F(x_1^*)=G(x_1^*,x_2^*)$.
\end{theorem}
One can also derive these equations by applying the Lagrange multiplier method to \eqref{FlowFun}. Note that there is no expectation of uniqueness for solutions to the optimality equations. Even if we are in conditions of Theorem \ref{OptPart} the solution obtained from the system above may be spurious, $\frac{\partial G}{\partial x_2}=0$ may pick out any critical point. Even if it is a global maximum for a given $x_1^*$, e.g. if $G$ is unimodal, it may not be the overall maximum. Even for linear $f,g$ such $x_1^*$ is a solution to a quadratic equation and hence not unique. Therefore, care has to be taken in the numerical procedure to ensure that the true maximum is found. If several overall maxima exist the solution with the minimal $x_1^*$ should be selected. 

Based on the above discussion we can formulate the following algorithm, where we assume that $f$ and $g$ satisfy the conditions of Lemma \ref{Unimod}.
\vspace{0.5em}

\noindent{\bf Algorithm:}
\vspace{-0.5em}
\begin{enumerate}
\item Solve $F'(x_1)=f(x_1)+f'(x_1)(x_1-x_0)=0$ to get $\widetilde{x}_1$.
\item Solve $\frac{\partial G}{\partial x_2}(\widetilde{x}_1,x_2)=g(x_2)+g'(x_2)(x_2-\widetilde{x}_1)=0$ to get $\widetilde{x}_2$.
\item If  $F(\widetilde{x}_1)\leq G(\widetilde{x}_1,\widetilde{x}_2)$ set $x_1^*=\widetilde{x}_1$ and solve $F(x_1^*)=G(x_1^*,x_2)$ to get $x_2^*$. Return $(x_1^*,x_2^*)$ (non-bottleneck case) and stop.
\item Otherwise, solve the system $\begin{cases}x_1=x_2+\frac{g(x_2)}{g'(x_2)}\\F(x_1)=G(x_1,x_2).\end{cases}$  Iterate on $[x_0,\widetilde{x}_2]$ to find the solution with the smallest possible $x_2$. Return the final $(x_1^*,x_2^*)$ (bottleneck case) and stop.
 \end{enumerate}

\section{Optimal solution for multiple segments}\label{S4}

Let $x_0\geq0$ and $K_i(x)$ be given, and set $G_i(x_{i-1},x_i)=(x_i-x_{i-1})K_i(x_i)$. Then the multi-segment version of problem \eqref{GenProb2} can be stated as follows.
\smallskip

\noindent {\bf Optimization problem.} Find $x_1,x_2,\cdots,x_n$ that maximize $G_n(x_{n-1},x_n)$ subject to the constraints
\begin{equation}\label{GenProb}
G_1(x_0,x_1)=G_2(x_1,x_2)=\cdots=G_n(x_{n-1},x_n)\,.
\end{equation}
The analysis is largely analogous to the $n=2$ case so we will simply state the result and sketch a proof.

\begin{theorem}\label{MultPart} Suppose $K_i(x)$ are non-negative and continuous, $K_i(x_0)>0$, and $xK_n(x)\xrightarrow[x\to\infty]{}0$. Then there exists a solution $(x_1^*,x_2^*,\cdots,x_n^*)$ to \eqref{GenProb} with a positive maximum and there exists a $1\leq k\leq n$ (bottleneck) such that $x_k^*$ is a local maximizer of $G_n(x_{k-1}^*,\cdot)$. If two solutions have the (first) bottleneck at $k$, and $G_k(c,x)$ is unimodal for any $c\geq0$, then they coincide up to it.
\end{theorem}
\begin{proof} By a continuity argument, we can select $x_1\leq x_2\leq\cdots\leq x_n$ so that $G_1(x_0,x_1)=G_2(x_1,x_2)=\cdots=G_n(x_{n-1},x_n)\geq\varepsilon>0$, and the optimization can be confined to the set defined by these constraints, which is closed. Since $G_n(x_{n-1},x_n)\leq x_nK_n(x_n)\xrightarrow[x_n\to\infty]{}0$ it is also bounded, so by the Weierstrass's theorem the maximum is attained on it.

Suppose that none of $x_1^*,x_2^*,\cdots,x_n^*$ maximizes the value on its segment, even locally. Then there is $\widetilde{x}_n$ such that $G_n(x_{n-1}^*,\widetilde{x}_n)>G_n(x_{n-1}^*,\widetilde{x}_n)$. Slight change of $x_{n-1}^*$ to $\widetilde{x}_{n-1}$ in either direction will still preserve the inequality. Since none of $x_i^*$ is a maximizer we can match a sufficiently small deviation on other segments and generate $\widetilde{x}_1,\cdots,\widetilde{x}_n$ such that 
$$
G_1(\widetilde{x}_0,\widetilde{x}_1)=\cdots=G_n(\widetilde{x}_{n-1},\widetilde{x}_n)>G_n(x_{n-1}^*,x_n^*), 
$$
which contradicts that $(x_1^*,x_2^*,\cdots,x_n^*)$ is a maximizer. Therefore, at least one $x_k^*$ maximizes $G_k(x_{k-1}^*,\cdot)$.

Consider solutions $x_i^*$ and $\widetilde{x}_i^*$ with the first bottleneck at $k$. We may assume without loss of generality that $\widetilde{x}_{k-1}^*\geq x_k^*$. Then 
\begin{multline*}
G_k(\widetilde{x}_{k-1}^*,\widetilde{x}_k^*)=(\widetilde{x}_k^*-\widetilde{x}_{k-1}^*)K_k(\widetilde{x}_k^*)
\leq(\widetilde{x}_k^*-x_{k-1}^*)K_k(\widetilde{x}_k^*)\\
\leq\max_{x_k}(x_k-x_{k-1}^*)K_k(x_k)=G_k(x_{k-1}^*,x_k^*)\,.
\end{multline*}
Since the first and the last values are by assumption equal both inequalities are equalities, and since 
$K_k(\widetilde{x}_k^*)$ must be positive we have $\widetilde{x}_{k-1}^*=x_{k-1}^*$. Since $k$ is a bottleneck and $G_k(x_{k-1}^*,\cdot)$ is unimodal we also have $\widetilde{x}_k^*=x_k^*$. Since the values of $G$ are equal and positive we have:
\begin{multline*}
G_{k-1}(\widetilde{x}_{k-2}^*,\widetilde{x}_{k-1}^*)=(\widetilde{x}_{k-1}^*-\widetilde{x}_{k-2}^*)K_{k-1}(\widetilde{x}_{k-1}^*)
=(x_{k-1}^*-\widetilde{x}_{k-2}^*)K_{k-1}(x_{k-1}^*)\\
=(x_{k-1}^*-x_{k-2}^*)K_{k-1}(x_{k-1}^*)=G_{k-1}(x_{k-2}^*,x_{k-1}^*)\,,
\end{multline*}
i.e. $\widetilde{x}_{k-2}^*=x_{k-2}^*$. By induction, this holds for all $i\leq k$. $\square$
\end{proof}
It follows that in the general case as well the optimal solution(s) can be found by solving systems obtained by appending $\frac{\partial G}{\partial x_i}=0$ for $i=1,\cdots,n$ to the constraint equations \eqref{GenProb}. This generically produces only finitely many potential solutions, whose flows should be compared to select the maximal one(s). Among those the one with smallest $x_i^*$ should be selected.

\section{Computations and conclusions}\label{S5}

It would be interesting to compare optimal distributions of water potentials in a variety of plants, but unfortunately, despite the vast literature on the subject, information on stem and leaf vulnerability curves is rarely reported for the same plant. Parameterizing a multi-segment model proved to be impossible given the available data. We were able to  parametrize a two-segment model for three eastern US tree species based on \cite{John}, and for {\it Helianthus annuus} (sunflower) based on \cite{Rico} and \cite{Sim}. Analytic fits, which were not given in the papers, were made according to the standard methodology \cite{Ogle}. As one can see from Table \ref{4Plant}, these species already display a wide range of stem to leaf maximal conductance ratios, from about $30$ in {\it Helianthus annuus} to about $1/30$ in {\it Pinus virginiana}. These formulas were used as inputs for applying the optimality model.
\begin{table}[!ht]
\centering
\label{my-label}
\setlength\extrarowheight{5pt}
\begin{tabular}{|l|l|l|}
\hline
               & $f(x)$ (stem) & $g(x)$ (leaf) 
            \\ \hline
{\it H. annuus}          & $11.9e^{-\left(\frac{x}{3.34}\right)^{1.69}}$
                    & $0.4(1-\frac{x}{1.64})$        
\\ \hline
{\it A. rubrum}           & $25.29e^{-\left(\frac{x}{4.22}\right)^{4.67}}$
                    & $29.2e^{-\left(\frac{x}{1.76}\right)^{10.24}}$            
\\ \hline
{\it L. tulipifera}      & $4.27e^{-\left(\frac{x}{3.26}\right)^{4.46}}$
                    & $9.8e^{-\left(\frac{x}{1.29}\right)^{4.91}}$               
\\ \hline
{\it P. virginiana} & $1.07e^{-\left(\frac{x}{4.59}\right)^{4.11}}$
                    & $32.8e^{-\left(\frac{x}{0.95}\right)^{2.15}}$                
\\ \hline
\end{tabular}
\caption{Stem and leaf vulnerability curves of four plant species. Values in ${\rm mmol\,m^{-2}\,s^{-1}\,MPa^{-1}}$ (normalized by the total leaf area) for the tree species, and in ${\rm mmol\,s^{-1}\,MPa^{-1}}$ (bulk) for {\it H. annuus}.}\label{4Plant}
\end{table}

We assumed that the soil potential is $\psi_0=0$. Computations were performed using mostly MATLAB's function fsolve, which requires specifying initial guesses for the potentials. The function was run iteratively with initial guesses equally spaced over admissible ranges of values estimated from the vulnerability curves. Fsolve did not work for P. virginiana, whose optimality system is numerically ill-conditioned due to almost constant stem vulnerability curve in the relevant range, Fig. \ref{PinVul}. Instead, we found the optimal values by using bisection on the leaf potential range.
\begin{figure}[!ht]
\begin{center}
\includegraphics[width = 0.7\textwidth]{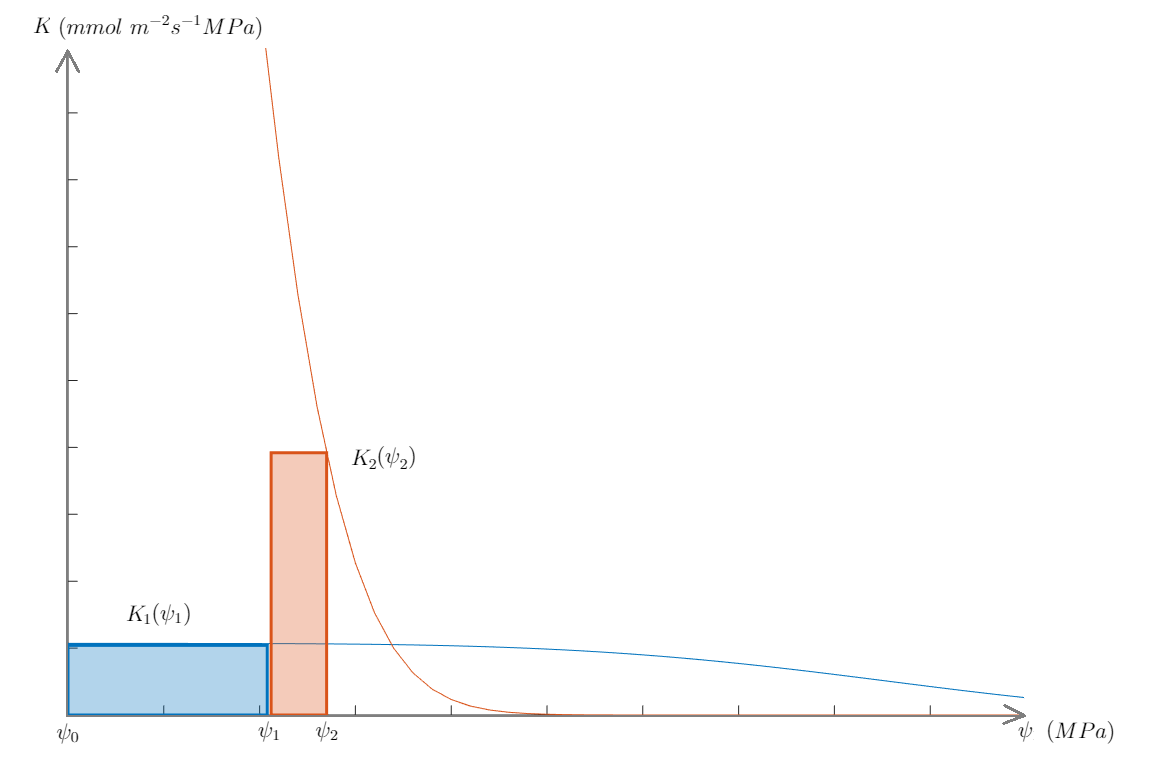}
\end{center}
\vspace{-0.05\textwidth}
\caption{\label{PinVul} Vulnerability curves of {\it P. virgniana}, and the optimal flow rectangles.}
\end{figure}

For each solution pair found flow values were compared and the maximal ones selected. The results of applying our optimization model are presented in Table \ref{4Pot}. In addition to the optimal values of the soil and the leaf water potentials we also list the overall optimal flow and the optimal flow through the stem taken in isolation to highlight the bottleneck effect.

One can see that in all cases the leaf segment served as the bottleneck for the flow. Note also that the optimal stem potentials are much lower than the values that would cause significant loss of conductance in the stem (roughly given by the values under $x$ in the exponents of Table \ref{4Plant}). 
\begin{table}[!ht]
\centering
\label{PotMax}
\setlength\extrarowheight{5pt}
\begin{tabular}{|l|l|l|l|l|}
\hline
         & $\psi^{{\rm max}}_S$ & $\psi^{{\rm max}}_L$ & $F^{{\rm max}}$ & $F_{{\rm stem}}^{{\rm max}}$
            \\ \hline
{\it H. annuus}          & 0.013  & 0.86
& 0.015  & 1.99 \\ \hline
{\it A. rubrum}          & 0.73  & 1.50
& 18.48  &  61.94\\ \hline
{\it L. tulipifera}      & 0.65  & 1.12       
& 2.78  & 7.96\\ \hline
{\it P. virginiana} & 1.06  & 1.35              
& 1.13  & 2.73 \\ \hline
\end{tabular}
\caption{Stem and leaf potentials (in MPa) that maximize the flow, the maximal flow, and the maximal flow through stems in isolation (per the total leaf area for the tree species).}\label{4Pot}
\end{table}

For the tree species \cite{John} reports the values of stem and leaf water potentials measured at midday. For {\it A. rubrum} they are $0.73$ and $1.53$ MPa, respectively, for {\it L. tulipifera} $0.65$ and $1.17$ MPa, and for {\it P. virginiana} $0.98$ and $1.56$ MPa. These are in a remarkably good agreement with the theoretically optimal values from Table \ref{4Pot}, especially considering how idealized the model assumptions were. 

Overall, our computations support the leaf safety buffer hypothesis, and suggest a better metric for testing it than comparing maximal conductances of or changes of potential across leaf and stem segments. As one can see from Table \ref{4Pot} the latter metrics may not detect the bottleneck behavior of leaves.

\bigskip
\noindent{\large\bf Acknowledgments:} This work was conceived during the summer 2017 REU program at the University of Houston-Downtown, and is funded by the National Science Foundation grant 1560401.


\begin{thebibliography}{100}

\bibitem{TZ} M. Tyree, M. Zimmermann, Xylem Structure and the Ascent of Sap, Springer-Verlag, Berlin Heidelberg New-York, 2002.

\bibitem{Sper} J. Sperry et al., Predicting stomatal responses to the environment from the optimization of photosynthetic gain and hydraulic cost, Plant, Cell \& Environment, 40 (2017) 816--830.

\bibitem{Wolf} A. Wolf et al., Optimal stomatal behavior with competition for water and risk of hydraulic impairment, PNAS, 113 (46) (2016) E7222--E7230.

\bibitem{Hol} N. Holbrook et al., In vivo observation of cavitation and embolism repair using magnetic resonance imaging, Plant Physiology, 126(1) (2001) 27--31. 

\bibitem{Wh} J. Wheeler et al., Cutting xylem under tension or supersaturated with gas can generate PLC and the appearance of rapid recovery from embolism, Plant, Cell and Environment, 36(11) (2013) 1938--1949.

\bibitem{ZwH} M. Zwieniecki, N. Holbrook, Confronting Maxwell's demon: biophysics of xylem embolism repair, Trends in Plant Science, 14 (2009) 530--534.

\bibitem{Vent} M. Venturas et al., Plant xylem hydraulics: What we understand, current research, and future challenges, Journal of Integrative Plant Biology, 59(6) (2017) 356--389. 

\bibitem{Drake} P. Drake et al., Isometric partitioning of hydraulic conductance between leaves and stems: balancing safety and efficiency in different growth forms and habitats, Plant, Cell and Environment, 38 (2015) 1628--1636 

\bibitem{Hao} G.-Y. Hao et al., Stem and leaf hydraulics of congeneric tree species from adjacent tropical savanna and forest ecosystems, Oecologia, 155 (2008) 405--415

\bibitem{John} D. Johnson et al., Hydraulic patterns and safety margins, from stem to stomata  in three eastern US tree species, Tree Physiology, 31 (2011) 659--668.

\bibitem{Sack} L. Sack et al., Stomatal control of xylem embolism, Plant, Cell and Environment, 14 (1991) 607--612.

\bibitem{Zim} M. Zimmerman, Hydraulic architecture of some diffuse porous trees, Canadian Journal of Botany 56 (1978) 2286--2295.

\bibitem{Scof1} C. Scoffoni et al., Outside-xylem vulnerability, not xylem embolism, controls leaf hydraulic decline during dehydration, Journal of Experimental Botany, 173 (2017)  1197--1210.

\bibitem{Coch} E. Cochard et al., Methods for measuring plant vulnerability to cavitation: a critical review,  Journal of Experimental Botany, 64 (2013) no. 15, 4779--4791.

\bibitem{Ogle} K. Ogle et al., Hierarchical statistical modeling of xylem vulnerability to cavitation, New Phytologist 182 (2009) 541--554.

\bibitem{JS} H. Jones, A. Sutherland, Stomatal control of xylem embolism, Plant, Cell and Environment (1991) 14, 607-612.

\bibitem{Uni} Unimodality, Convexity, and Applications, S. Dharmadhikari and K. Joag-Dev, editors, Academic Press, 1988.

\bibitem{Rico} C. Rico et al., The effect of subambient to elevated atmospheric CO$_2$ concentration on vascular function in Helianthus annuus: implications for plant response to climate change, New Phytologist, 199 (2013) 956--965.

\bibitem{Sim} K. Simonin et al., Increasing leaf hydraulic conductance with transpiration rate minimizes the water potential drawdown from stem to leaf, Journal of Experimental Botany 66(5) (2015) 1303--1315.


\bibitem{Scof} C. Scoffoni et al., Dynamics of leaf hydraulic conductance with water status: quantification and analysis of species differences under steady state, Journal of Experimental Botany, 63(2) (2012) 643--58. 
























\end{thebibliography}
\end{document}